\documentclass[11pt,reqno]{amsart}
\usepackage{graphicx}
\usepackage[colorlinks=true, urlcolor=blue, linkcolor=blue, citecolor=blue]{hyperref} \usepackage{amsmath,amssymb,amsfonts,euscript,enumerate}
\usepackage{amsthm}  
\usepackage{color}
\usepackage{verbatim}
\usepackage{latexsym}
\usepackage[nameinlink]{cleveref}
\usepackage{mathrsfs} 
\usepackage{mathtools}

\usepackage{booktabs}
\usepackage{tabularx}
\usepackage{array}

\usepackage[dvipsnames]{xcolor}    

\newcommand{\R}{\mathbb{R}}
\newcommand{\PV}{\mathrm{PV}}

\newcommand{\supp}{\operatorname{supp}}

\newtheorem{theorem}{Theorem}[section]
\newtheorem{lemma}[theorem]{Lemma}
\newtheorem{proposition}[theorem]{Proposition}
\newtheorem{corollary}[theorem]{Corollary}

\theoremstyle{plain}

\theoremstyle{definition}
\newtheorem{definition}[theorem]{Definition}

\theoremstyle{remark}
\newtheorem{remark}[theorem]{Remark}

\numberwithin{equation}{section}

\title[Liouville theorems with degenerate kernels]
{Liouville theorems for nonlocal Lane--Emden inequalities with degenerate kernels}

\author{Takwon Kim}
\address{School of Mathematics, Statistics and Data Science, Sungshin Women’s University, Seoul 02844, Republic of Korea.}
\email{takwon@sungshin.ac.kr}

\author{Taehun Lee}
\address{Department of Mathematics, Konkuk University, Seoul 05029, Republic of Korea.}
\email{taehun@konkuk.ac.kr}



\begin{document} 

\begin{abstract}
We establish a Liouville-type theorem for the inequality $\mathcal{L}_K u \ge |u|^q$ in $\mathbb{R}^n$, where $\mathcal{L}_K$ is a translation-invariant, degenerate elliptic integro-differential operator whose kernel is even and satisfies $0\le K(z)\le \Lambda|z|^{-n-2s}$. No lower ellipticity bound is imposed on the kernel, and no sign condition is imposed on the solution. We prove that every such solution is trivial when $1 \le q \le \frac{n}{n-2s}$ if $n > 2s$, and for every $q \ge 1$ if $n \le 2s$. The proof combines a test function method with a dyadic decomposition of the nonlocal tail and uses neither the maximum principle nor a fundamental solution.
\end{abstract}

\maketitle

\section{Introduction}

Liouville-type theorems play a pivotal role in the qualitative analysis of semilinear elliptic equations. They assert that, under suitable restrictions on the nonlinearity, any nonnegative solution (or supersolution) in the whole space is trivial. Such nonexistence results are a key ingredient in the method of moving planes, blow-up analysis, and the derivation of a priori bounds via rescaling arguments; see the seminal work of Gidas--Spruck \cite{GS81_CPAM} and the monograph \cite{QS19_book}.

In this paper, we consider the inequality
\begin{equation}\label{eq:main-ineq}
\mathcal L_Ku\ge |u|^q
\qquad\text{in }\mathbb R^n,
\end{equation}
where \(\mathcal L_K\) is the translation-invariant nonlocal operator given by
\begin{equation}\label{eq:LK}
\mathcal L_Ku(x)
:=
\mathrm{P.V.}\int_{\mathbb R^n}
\bigl(u(x)-u(x+z)\bigr)K(z)\,dz.
\end{equation}
Here \(K:\mathbb R^n\setminus\{0\}\to[0,\infty)\) is measurable and satisfies
\begin{equation}\label{eq:K-assump}
0\le K(z)\le\Lambda |z|^{-n-2s},
\qquad
K(z)=K(-z)
\qquad (z\ne0)
\end{equation}
for some \(\Lambda>0\).
No lower ellipticity bound is imposed on \(K\), so the kernel may be degenerate. In the uniformly elliptic case \(\lambda|z|^{-n-2s}\le K(z)\), operators of this form arise as generators of stable jump processes, and the regularity theory of the associated equations is well understood; see \cite{CS09_CPAM, CS11_Annals, CS11_ARMA}.
The inequality \eqref{eq:main-ineq} is understood in
the distributional sense specified in Definition~\ref{def:weak_sol}.
We assume that
\begin{equation}\label{eq:q-range}
1 \le q \le q_S := \frac{n}{n-2s} \quad \text{if } n > 2s, \qquad 1 \le q < \infty \quad \text{if } n \le 2s.
\end{equation}

The main result of this paper is the following.
\begin{theorem}\label{thm:main}
Let \(n\ge1\), \(s\in(0,1)\), and let \(K\) satisfy
\eqref{eq:K-assump}. Assume that \(q\) satisfies
\eqref{eq:q-range}. Then any weak supersolution $u$ of \eqref{eq:main-ineq} is trivial, namely, $u\equiv 0$ in $\mathbb{R}^n$.
\end{theorem}

When $n > 2s$, the exponent $q_S = \frac{n}{n-2s}$ in \eqref{eq:q-range} is the natural Serrin-type threshold for the existence of nontrivial supersolutions, and Theorem~\ref{thm:main} fails for every $q>q_S$. Indeed, as the class \eqref{eq:K-assump} contains the fractional Laplacian $\mathcal{L}_K = (-\Delta)^s$, this follows from the fact that for every $q > q_S$, the function
\begin{align}
u(x) = c(1+|x|)^{-\frac{2s}{q-1}}
\end{align}
is a weak supersolution of $(-\Delta)^s u \ge u^q$ in $\mathbb{R}^n$ for a sufficiently small constant $c > 0$, see \cite{FQ11_AM}.

The borderline case $q = 1$ is included in \eqref{eq:q-range}. In the local case $s = 1$, the corresponding statement fails without a growth restriction: the function $u(x) = -e^{x_1}$ solves $-\Delta u = |u|$ in $\mathbb{R}^n$. In our setting, such exponentially growing solutions are excluded by the tail condition $u \in L^1_{2s}(\mathbb{R}^n)$ in Definition~\ref{def:weak_sol} which is needed to define weak solutions (see Remark~\ref{rem:Lkphi-decay}), and which provides the initial polynomial bound of Lemma~\ref{lem:local_est} from which the iteration starts.

For the equation $(-\Delta)^s u = u^q$ in $\mathbb{R}^n$, the relevant threshold is the fractional Sobolev exponent $\frac{n+2s}{n-2s}$, and nonexistence results up to this exponent have been obtained through the extension method and moving plane and moving sphere arguments \cite{Li04_JEMS, CLO06_CPAM, CFY15_AM, CLZ17_JFA, CLL17_AM}; see also \cite{DDW17_TAMS} for solutions with finite Morse index. For the inequality $(-\Delta)^s u \ge u^q$, the threshold is the strictly smaller Serrin exponent $q_S$, and the Liouville theorem in the full range $1 < q \le q_S$ has been proved by barrier arguments; see \cite{CHW25_TAMS, BQ25_N} and the references therein. These arguments are based on the decay of the fundamental solution $\Phi(x) = c_{n,s}|x|^{-(n-2s)}$ and on the comparison principle. Liouville theorems for supersolutions are also of separate interest because blow-up procedures often produce limits that satisfy only an inequality.

Beyond the fractional Laplacian, Felmer and Quaas \cite{FQ11_AM} proved Liouville theorems for the extremal inequality $\mathcal{M}^- u + u^q \le 0$ in $\mathbb{R}^n$ in the range $1 < q \le \frac{N^-}{N^- - 2s}$, where the effective dimension $N^- \in [n, n+2s)$ comes from the nonlinearity of the Pucci operator $\mathcal{M}^-$ and the threshold lies strictly below $q_S$. Their proof again uses the fundamental solution of $\mathcal{M}^-$ and the maximum principle, and the kernels of the linear operators defining $\mathcal{M}^-$ are comparable to that of the fractional Laplacian.

Degenerate kernels were recently considered by Birindelli, Du, and Galise \cite{BDG25_CVPDE}. They studied stable operators with kernel of the form $a(z/|z|)\,|z|^{-n-2s}$,
where $a \in L^\infty(\mathbb{S}^{n-1})$ is even, nonnegative, and bounded below
by a positive constant on some cone of directions, so that the operator may diffuse only along that cone. For nonnegative classical supersolutions they proved the Liouville theorem in the sharp range $1 \le q \le \frac{n+s}{n-s}$ in the half-space, and in the range \eqref{eq:q-range} in the whole space, by a test function method adapted to this weak diffusion. In the whole space, their argument is based on the pointwise bound $\mathcal{L}\varphi_R \le C R^{-2s}\varphi_R$ for rescaled cutoff functions, which produces no tail term. The proof of this bound uses the positivity of $a(z/|z|)$ on a cone. Theorem~\ref{thm:main} extends their whole space result to kernels that are merely measurable with the one-sided bound \eqref{eq:K-assump}, with no homogeneity and no positivity on any cone, and to solutions without a sign condition.

The main point of the present paper is that only the upper bound in \eqref{eq:K-assump} is used. As a consequence, the result covers inhomogeneous, anisotropic, and degenerate kernels, for which a fundamental solution and a Harnack inequality need not exist and the comparison principle may fail, and the proof uses none of these tools. Moreover, the argument applies to signed supersolutions, and the range \eqref{eq:q-range} includes the borderline case $q = 1$.

Our approach is the test function method for nonexistence results, which goes back in the local setting to Mitidieri--Pohozaev \cite{MP01_TMIS} and Bidaut-V\'eron--Pohozaev \cite{BP01_JAM}. In the nonlocal setting, the test function $\varphi_R$ interacts with $u$ over the whole $\mathbb{R}^n$, and this produces the tail term $\int_{\mathbb{R}^n \setminus B_{2R}} u \, \mathcal{L}_K \varphi_R \, dx$ with no local analogue. We control this term by a dyadic decomposition. Partitioning $\mathbb{R}^n \setminus B_{2R}$ into the annuli $B_{2^{k+1}R} \setminus B_{2^k R}$ and applying H\"older's inequality on each annulus leads to the recursive inequality
\begin{align}
S(R) \le C \, R^{a} \sum_{k=0}^{\infty} 2^{-kb} \, S(2^{k+1} R)^{1/q}
\end{align}
for the local mass $S(R) := \int_{B_R}|u|^q \, dx$, where $a = n - 2s - n/q$ and $b = 2s + n/q$ (Lemma~\ref{lem:recursive}). If $q < q_S$, then $a < 0$, and iterating the inequality lowers the growth exponent of $S(R)$ until it becomes negative, which forces $S(R) \to 0$. If $q = q_S$, then $a = 0$, and the same iteration gives $u \in L^q(\mathbb{R}^n)$, and a tail estimate concludes the proof.

Although we focus on the operator $\mathcal{L}_K$ with a kernel satisfying \eqref{eq:K-assump}, the proof of Theorem~\ref{thm:main} uses only the weak formulation and the decay estimate of Lemma~\ref{lem:cutoff-K} for the operator applied to cutoff functions. The argument therefore applies to any operator $\mathcal{L}$ whose formal adjoint satisfies a similar estimate, where a weak supersolution of $\mathcal{L} u \ge |u|^q$ is understood as in Definition~\ref{def:weak_sol} with $\mathcal{L}_K$ replaced by $\mathcal{L}^\ast$. In the statement below, $\varphi_R(x) = \varphi(x/R)$ is the standard rescaling of a smooth cutoff function $\varphi$ such that $\varphi \equiv 1$ in $B_1$ and $\text{supp}(\varphi) \subset B_2$.

\begin{proposition}\label{cor:sub}
Let $n \ge 1$, $s \in (0,1)$, and let $\mathcal{L}$ be an operator whose formal adjoint $\mathcal{L}^\ast$ satisfies
\begin{equation}
    |\mathcal{L}^\ast (\varphi_R)(x) | \le \frac{C}{R^{2s}}\mathbf{1}_{B_{2R}}(x)
    + C\,\frac{R^n}{|x|^{n+2s}}\,\mathbf{1}_{\mathbb{R}^n\setminus B_{2R}}(x)
\end{equation}
for every $R\ge1$, with a constant $C>0$ independent of $R$. Assume that $q$ satisfies \eqref{eq:q-range}. Then any weak supersolution $u$ of
\begin{equation}
    \mathcal{L} u \ge |u|^q \quad \text{in } \mathbb{R}^{n}
\end{equation}
 is trivial, namely, $u \equiv 0$ in $\mathbb{R}^n$.
\end{proposition}

\begin{remark}\label{rem:signed}
The positivity of $K$ plays no role in our argument. Theorem~\ref{thm:main} remains valid for every measurable even kernel satisfying only $|K(z)|\le \Lambda\,|z|^{-n-2s}$ for $z\ne0$, so that $K$ may change sign or be nonpositive. Indeed, for any such $K$, the operator $\mathcal{L}_K$ coincides with its formal adjoint since $K$ is even, and it satisfies the cutoff estimate of Lemma~\ref{lem:cutoff-K}, so Proposition~\ref{cor:sub} applies.
\end{remark}

Proposition~\ref{cor:sub} also applies to mixed local-nonlocal operators, which have recently attracted considerable attention; see, for instance, \cite{BDVV22,BDVV25_NoDEA}. For the prototype $-\Delta+(-\Delta)^s$, Guo, Li, and Xie \cite{GLX26} established the sharp existence-nonexistence dichotomy for positive distributional supersolutions. The following consequence covers finite linear combinations of the Laplacian and fractional Laplacians, with no sign restriction on the coefficients or on the solutions.

\begin{corollary}[Mixed-order and local perturbations]
\label{cor:mixed-orders}
Let
\begin{align}
0<s_1<\cdots<s_m<1,
\qquad
c_1\neq0,
\qquad
c_2,\ldots,c_m, c_0\in\mathbb R,
\end{align}
and set
\begin{align}
\mathcal L
:=
c_0(-\Delta)+\sum_{j=1}^m c_j(-\Delta)^{s_j}.
\end{align}
Assume that \(q\) satisfies \eqref{eq:q-range} with \(s_1\). Then any function \(u\) satisfying
\begin{align}
\mathcal Lu\ge |u|^q
\qquad\text{in }\mathbb R^n
\end{align}
in the sense of Definition~\ref{def:weak_sol}, with
\(\mathcal L_K\) replaced by \(\mathcal L\) and \(s\) by \(s_1\), is
trivial.
\end{corollary}

\medskip

The paper is organized as follows. In Section~\ref{sec:pre}, we introduce
the weak formulation and derive an estimate for the operator applied to
cutoff functions. Section~\ref{sec:esti_iter} establishes a polynomial growth
bound and a dyadic recursive inequality for the local mass \(S(R)\).
Finally, Section~\ref{sec:pf} proves Theorem~\ref{thm:main} and
Corollary~\ref{cor:mixed-orders}.

\section{Preliminaries}\label{sec:pre}

From now on, the nonlocal operator $\mathcal L_K$ is defined by \eqref{eq:LK}
with a kernel $K$ satisfying \eqref{eq:K-assump}, that is,
\begin{align}
    0 \le K(z)\le \Lambda |z|^{-n-2s},\qquad K(z)=K(-z)\qquad(z\neq 0).
\end{align}
In particular, $\mathcal L_K$ is translation-invariant and self-adjoint,
and degenerate since no lower bound on $K$ is imposed.

\subsection*{Tail space and weak supersolutions}

Since $\mathcal L_K$ is nonlocal, we impose a mild integrability condition at infinity.
Define the tail space
\begin{align}
L_{2s}^1(\R^n)
:=\left\{u\in L^1_{\rm loc}(\R^n):\ \int_{\R^n}\frac{|u(x)|}{(1+|x|)^{n+2s}}\,dx<\infty\right\}.
\end{align}

\begin{definition}\label{def:weak_sol}
Let $q\ge1$. A function $u$ is called a \emph{weak supersolution}
to \eqref{eq:main-ineq} if
$u\in L^q_{\rm loc}(\R^n)\cap L_{2s}^1(\R^n)$
and
\begin{align}\label{eq:weak-supersol}
\int_{\R^n} u(x)\,\mathcal L_K\varphi(x)\,dx
\;\ge\;
\int_{\R^n} |u(x)|^q\,\varphi(x)\,dx
\end{align}
for every test function $\varphi\in C_c^\infty(\R^n)$ with $\varphi\ge 0$.
\end{definition}

\begin{remark}\label{rem:Lkphi-decay}
If $\varphi\in C_c^\infty(\R^n)$, then $\mathcal L_K\varphi(x)$ decays like $|x|^{-n-2s}$ as $|x|\to\infty$. Indeed, if $\supp\varphi\subset B_{2R}$ and $|x|\ge 4R$, then $\varphi(x)=0$ and
\begin{align}
|\mathcal L_K\varphi(x)|
=\left|\int_{\R^n}\varphi(x+z)K(z)\,dz\right|
=\left|\int_{B_{2R}}\varphi(y)\,K(y-x)\,dy\right|
\le C\,\frac{R^n}{|x|^{n+2s}},
\end{align}
where the last inequality follows from the upper bound \eqref{eq:K-assump} on $K$ and the fact that $|y-x|\ge |x|-2R\ge \frac{|x|}{2}$ for $y\in B_{2R}$; see also \eqref{est:outside}.
Therefore, the tail condition $u\in L_{2s}^1(\R^n)$ guarantees that the left-hand side of \eqref{eq:weak-supersol} is finite. On the right-hand side, the term $\int |u|^q \varphi$ is well-defined since $u \in L^q_{\rm loc}(\mathbb{R}^n)$ and $\varphi$ has compact support.
\end{remark}

\subsection*{Cutoff functions}

Fix $\varphi\in C_c^\infty(\R^n)$ satisfying
\begin{equation}\label{def:eta}
0\le \varphi\le 1,\qquad \varphi\equiv 1 \ \text{in }B_1,\qquad \varphi\equiv 0 \ \text{in }\R^n\setminus B_2.
\end{equation}
For $R\ge 1$ define $\varphi_R(x):=\varphi(x/R)$.

\begin{lemma}[Cutoff estimate]\label{lem:cutoff-K}
There exists $C=C(n,s,\Lambda, \varphi)>0$ such that for every $R\ge 1$ and every $x\in\R^n$,
\begin{equation}\label{eq:cutoff-lemma}
\bigl|\mathcal L_K(\varphi_R)(x)\bigr|
\le \frac{C}{R^{2s}}\mathbf 1_{B_{2R}}(x)
+ C\,\frac{R^n}{|x|^{n+2s}}\,\mathbf 1_{\R^n\setminus B_{2R}}(x).
\end{equation}
\end{lemma}

\begin{proof}
Fix $x\in\R^n$. Using $K(z)=K(-z)$, we symmetrize the principal value integral:
\begin{align}\label{eq:symmetrize}
\mathcal L_K\varphi_R(x)
&=\PV\int_{\R^n}\big(\varphi_R(x)-\varphi_R(x+z)\big)\,K(z)\,dz \\
&=\frac12\int_{\R^n}\big(2\varphi_R(x)-\varphi_R(x+z)-\varphi_R(x-z)\big)\,K(z)\,dz.
\end{align}
The integrand in \eqref{eq:symmetrize} is $O(|z|^2)$ as $z\to 0$, hence the last integral is absolutely
convergent.

\smallskip
\noindent\emph{Step 1: estimate on $B_{2R}$.}
For $|z|\le R$, Taylor's theorem yields
\begin{align}
\big|2\varphi_R(x)-\varphi_R(x+z)-\varphi_R(x-z)\big|
\le \|D^2\varphi_R\|_{L^\infty}\,|z|^2.
\end{align}
Since $\varphi_R(x)=\varphi(x/R)$, one has $\|D^2\varphi_R\|_{L^\infty}\le C R^{-2}$.
Using the upper bound $K(z)\le \Lambda |z|^{-n-2s}$,
\begin{align}
\int_{|z|\le R}\big|2\varphi_R(x)-\varphi_R(x+z)-\varphi_R(x-z)\big|\,K(z)\,dz
\\
\le C R^{-2}\int_{|z|\le R}|z|^{2-n-2s}\,dz
\le C R^{-2s}.
\end{align}
For $|z|>R$ we use $|2\varphi_R(x)-\varphi_R(x+z)-\varphi_R(x-z)|\le 4$ to obtain
\begin{align}
\int_{|z|>R}\big|2\varphi_R(x)-\varphi_R(x+z)-\varphi_R(x-z)\big|\,K(z)\,dz
\\
\le C\int_{|z|>R}|z|^{-n-2s}\,dz
\le C R^{-2s}.
\end{align}
Combining with \eqref{eq:symmetrize} gives
\begin{equation}\label{eq:uniform-cutoff}
|\mathcal L_K\varphi_R(x)|\le C R^{-2s}\qquad\text{for all }x\in\R^n,
\end{equation}
and in particular this yields the first term in \eqref{eq:cutoff-lemma} on $B_{2R}$.

\smallskip
\noindent\emph{Step 2: estimate on $\R^n\setminus B_{2R}$.}
If $|x|\ge 2R$, then $\varphi_R(x)=0$, and thus
\begin{align}\label{est:outside}
|\mathcal L_K\varphi_R(x)|
&=\left|\int_{\R^n}\varphi_R(x+z)K(z)\,dz\right|
=\left|\int_{B_{2R}}\varphi_R(y)\,K(y-x)\,dy\right|
\\&\le \Lambda\int_{B_{2R}}|x-y|^{-n-2s}\,dy.
\end{align}
If $|x|\ge 4R$, then $|x-y|\ge |x|/2$ for $y\in B_{2R}$, hence
\begin{align}
|\mathcal L_K\varphi_R(x)|\le C\,\frac{R^n}{|x|^{n+2s}}.
\end{align}
If $2R\le |x|\le 4R$, we use \eqref{eq:uniform-cutoff} together with
$R^{-2s}\le C R^n |x|^{-n-2s}$ (since $|x|\le 4R$) to obtain the same bound.
This proves \eqref{eq:cutoff-lemma}.
\end{proof}

\section{$L^q$-growth estimates and a dyadic inequality}\label{sec:esti_iter}

Let $u$ be a weak supersolution of \eqref{eq:main-ineq}. For $R\ge 1$ we set
\begin{equation}\label{eq:SR-def}
S(R):=\int_{B_R}|u|^q\,dx.
\end{equation}
In this section, we first obtain a rough polynomial bound for $S(R)$ and then establish a dyadic recursive inequality. This inequality serves as the starting point for the iteration argument presented in Section~\ref{sec:pf}.

\begin{lemma}\label{lem:local_est}
Let $q\ge1$ and let $u$ be a weak supersolution of \eqref{eq:main-ineq}. Then there exists $C>0$, independent of $R$, such that
\begin{equation}\label{eq:local_bound}
S(R)=\int_{B_R}|u|^q\,dx \le C\,R^n \qquad\text{for all }R\ge 1.
\end{equation}
\end{lemma}

\begin{proof}
Since $\varphi_R\equiv 1$ on $B_R$ and $\varphi_R\ge 0$, the weak formulation
\eqref{eq:weak-supersol} yields
\begin{align}
S(R)\le \int_{\R^n}|u|^q\varphi_R\,dx \le \int_{\R^n}u\,\mathcal L_K\varphi_R\,dx
\le \int_{\R^n}|u|\,|\mathcal L_K\varphi_R|\,dx.
\end{align}
 Applying Lemma~\ref{lem:cutoff-K}, we obtain
\begin{equation}\label{eq:SR-basic-split}
S(R)\le \frac{C}{R^{2s}}\int_{B_{2R}}|u|\,dx
+ C R^n\int_{\R^n\setminus B_{2R}}\frac{|u(x)|}{|x|^{n+2s}}\,dx.
\end{equation}
Let
\begin{align}
\mathcal T(u):=\int_{\R^n}\frac{|u(x)|}{(1+|x|)^{n+2s}}\,dx<\infty.
\end{align}
Since $R\ge 1$ and $|x|\ge 2R$ implies $|x|\ge \frac12(1+|x|)$, we have
\begin{align}
\int_{\R^n\setminus B_{2R}}\frac{|u(x)|}{|x|^{n+2s}}\,dx
\le C\int_{\R^n\setminus B_{2R}}\frac{|u(x)|}{(1+|x|)^{n+2s}}\,dx
\le C\,\mathcal T(u).
\end{align}
Moreover, for $x\in B_{2R}$ we have $(1+|x|)^{n+2s}\le (1+2R)^{n+2s}\le C R^{n+2s}$, hence
\begin{align}
\int_{B_{2R}}|u|\,dx
\le C R^{n+2s}\int_{B_{2R}}\frac{|u(x)|}{(1+|x|)^{n+2s}}\,dx
\le C R^{n+2s}\,\mathcal T(u).
\end{align}
Substituting the last two bounds into \eqref{eq:SR-basic-split} gives
\begin{align}
S(R)\le C R^{-2s}\cdot R^{n+2s}\mathcal T(u) + C R^n\mathcal T(u)\le C R^n,
\end{align}
which proves \eqref{eq:local_bound}.
\end{proof}

\begin{lemma}\label{lem:recursive}
Let $u$ be a weak supersolution of \eqref{eq:main-ineq} and let $S(R)$ be given by
\eqref{eq:SR-def}. Then there exists $C>0$ such that for every $R\ge 1$,
\begin{equation}\label{eq:recursive_ineq}
S(R)\le C\,R^{a}\sum_{k=0}^{\infty}2^{-kb}\,S(2^{k+1}R)^{1/q},
\end{equation}
where
\begin{equation}\label{eq:ab-def}
a:=n-2s-\frac{n}{q},\qquad b:=2s+\frac{n}{q}.
\end{equation}
\end{lemma}

\begin{proof}
As in the proof of Lemma~\ref{lem:local_est},
\begin{align}
S(R)\le \int_{\R^n}|u|\,|\mathcal L_K\varphi_R|\,dx.
\end{align}
Using Lemma~\ref{lem:cutoff-K} we obtain
\begin{equation}\label{eq:SR-start}
S(R)\le \frac{C}{R^{2s}}\int_{B_{2R}}|u|\,dx
+ C R^n\int_{\R^n\setminus B_{2R}}\frac{|u(x)|}{|x|^{n+2s}}\,dx.
\end{equation}
By H\"older's inequality,
\begin{equation}\label{eq:local-holder}
\int_{B_{2R}}|u|\,dx
\le |B_{2R}|^{1-1/q}\Bigl(\int_{B_{2R}}|u|^q\,dx\Bigr)^{1/q}
\le C R^{n(1-1/q)}S(2R)^{1/q}.
\end{equation}

For $k\ge 1$ define the annuli
\begin{align}
A_k:=B_{2^{k+1}R}\setminus B_{2^kR}.
\end{align}
Then $|x|\ge 2^kR$ on $A_k$, and therefore
\begin{align}\label{eq:tail-dyadic}
\int_{\R^n\setminus B_{2R}}\frac{|u(x)|}{|x|^{n+2s}}\,dx
&=\sum_{k=1}^{\infty}\int_{A_k}\frac{|u(x)|}{|x|^{n+2s}}\,dx\\
&\le \sum_{k=1}^{\infty}(2^kR)^{-(n+2s)}\int_{A_k}|u|\,dx.
\end{align}
Applying H\"older on each $A_k$ and using $|A_k|\le C(2^kR)^n$, we find
\begin{align}
\int_{A_k}|u|\,dx
&\le |A_k|^{1-1/q}\Bigl(\int_{A_k}|u|^q\,dx\Bigr)^{1/q}
\\
&\le C(2^kR)^{n(1-1/q)}\,S(2^{k+1}R)^{1/q}.
\end{align}
Substituting this into \eqref{eq:tail-dyadic} yields
\begin{equation}\label{eq:tail-final}
\int_{\R^n\setminus B_{2R}}\frac{|u(x)|}{|x|^{n+2s}}\,dx
\le C R^{-2s-\frac{n}{q}}\sum_{k=1}^{\infty}2^{-k(2s+\frac{n}{q})}\,S(2^{k+1}R)^{1/q}.
\end{equation}

Combining \eqref{eq:SR-start}, \eqref{eq:local-holder}, and \eqref{eq:tail-final} we obtain
\begin{align}
S(R)\le C R^{n-2s-\frac{n}{q}}
\Bigl(S(2R)^{1/q}+\sum_{k=1}^{\infty}2^{-k(2s+\frac{n}{q})}S(2^{k+1}R)^{1/q}\Bigr),
\end{align}
which is exactly \eqref{eq:recursive_ineq} after reindexing the sum and recalling \eqref{eq:ab-def}.
\end{proof}

\begin{remark}\label{rem:a-negative-subcritical}
When $n>2s$, the exponent $a=n-2s-n/q$ in \eqref{eq:ab-def} is negative if and only if $q$ is in the subcritical range $q<n/(n-2s)$. In this range, one has $R^{a}\to 0$ as $R\to\infty$.
This decay factor is the key in the iteration based on \eqref{eq:recursive_ineq} in Section~\ref{sec:pf}.
\end{remark}

\section{Proofs of the main results}\label{sec:pf}

Based on the estimates established in
Section~\ref{sec:esti_iter}, we first prove
Theorem~\ref{thm:main} by considering the subcritical, critical, and
low-dimensional cases separately. We then prove
Corollary~\ref{cor:mixed-orders}.

\subsection{The subcritical case: $1 \le q < \frac{n}{n-2s}$}

In this case, the exponent $a$ defined in Lemma \ref{lem:recursive} satisfies
$$
a = n - 2s - \frac{n}{q}  < 0.
$$
We use an iteration argument using the recursive inequality \eqref{eq:recursive_ineq} to improve the polynomial growth rate of $S(R)$.

Let $\gamma_0 = n$. By Lemma \ref{lem:local_est}, we have the initial bound $S(R) \le C_0 R^{\gamma_0}$ for all $R \ge 1$. We proceed by induction. Suppose that at step $m \ge 0$, the estimate
\begin{align}\label{eq:ind_hyp}
S(R) \le C_m R^{\gamma_m}, \quad \text{for all } R \ge 1
\end{align}
holds for some constant $C_m > 0$. Substituting this into the recursive inequality \eqref{eq:recursive_ineq}, we estimate the infinite series term:
\begin{align}
\sum_{k=0}^{\infty} 2^{-kb} S(2^{k+1}R)^{\frac{1}{q}} 
&\le \sum_{k=0}^{\infty} 2^{-kb} \left[ C_m (2^{k+1}R)^{\gamma_m} \right]^{\frac{1}{q}} \\
&= C_m^{\frac{1}{q}} R^{\frac{\gamma_m}{q}} 2^{\frac{\gamma_m}{q}} \sum_{k=0}^{\infty} 2^{-k(b - \frac{\gamma_m}{q})}.
\end{align}
Recall that $b = 2s + n/q$. Since we start with $\gamma_0 = n$ and the sequence $\gamma_m$ is decreasing (as shown below), the exponent in the sum satisfies
$$
b - \frac{\gamma_m}{q} \ge b - \frac{n}{q} = 2s > 0.
$$
Thus, the geometric series converges to a finite constant $\sigma_m$, which is uniformly bounded by some $\sigma < \infty$ independent of $m$. Combining this with \eqref{eq:recursive_ineq}, we obtain
\begin{align}
S(R) \le C R^a \cdot C_m^{\frac{1}{q}} R^{\frac{\gamma_m}{q}} 2^{\frac{\gamma_m}{q}} \sigma_m 
= \left( C 2^{\frac{\gamma_m}{q}} \sigma_m C_m^{\frac{1}{q}} \right) R^{a + \frac{\gamma_m}{q}}.
\end{align}
This yields the following recurrence relations for the exponent and the constant:
\begin{align}\label{eq:recurrence}
\gamma_{m+1} := a + \frac{\gamma_m}{q}, \qquad C_{m+1} := \bar{C} C_m^{\frac{1}{q}},
\end{align}
where $\bar{C}$ is a constant depending on $n, s, q$ but independent of $R$ and $m$.

If $q > 1$, the recurrence for the exponent has the fixed point $\gamma_\infty = \frac{a}{1-1/q} < 0$, and the sequence $\gamma_m$ decreases to $\gamma_\infty$. If $q = 1$, then $\gamma_{m+1} = \gamma_m - 2s$, so that $\gamma_m = n - 2ms$. In either case, there exists a finite integer $M$ such that $\gamma_M < 0$.

Regarding the constant, iterating the relation $C_{m+1} = \bar{C}\, C_m^{1/q}$ gives
\begin{align}
C_M = \bar{C}^{\,1 + \frac{1}{q} + \cdots + \frac{1}{q^{M-1}}}\, C_0^{\,q^{-M}} < \infty.
\end{align}
Consequently, after $M$ iterations, we arrive at the bound
\begin{align}
    S(R) \le C_M R^{\gamma_M}.
\end{align}
Since $\gamma_M < 0$, taking the limit $R \to \infty$ implies $\int_{\R^n} |u|^q \, dx = 0$. Therefore we conclude that $u \equiv 0$ in $\R^n$.

\subsection{The critical case: $q = \frac{n}{n-2s}$}

We now consider the critical case. The dyadic inequality \eqref{eq:recursive_ineq} still holds and yields
\begin{align}
    \int_{B_R}|u|^q\, dx \le \left( C \cdot 2^{\frac{\gamma_m}{q}} \sigma \cdot C_m^{\frac{1}{q}} \right) R^{\frac{\gamma_m}{q}},
\end{align}
where
\begin{align}
    \gamma_{m+1} = \frac{\gamma_m}{q},
    \qquad
    C_{m+1} = \bar{C}\, C_m^{\frac{1}{q}}.
\end{align}
Since $\{C_m\}$ is bounded and $\gamma_m \to 0$ as $m\to\infty$, we may pass to the limit and obtain
\begin{equation}
    \int_{B_R}|u|^q\, dx
    \le \lim_{m \to \infty} \left( C \cdot 2^{\frac{\gamma_m}{q}} \sigma \cdot C_m^{\frac{1}{q}} \right) R^{\frac{\gamma_m}{q}}
    \le C \sigma \lim_{m \to \infty} C_m^{\frac{1}{q}} .
\end{equation}
The right-hand side is independent of $R\ge 1$. Hence $u \in L^q(\mathbb{R}^n)$.

We next show $u \equiv 0$. Fix $\rho \in (0,R)$. By the weak formulation and the choice of $\varphi_R$,
\begin{align}
    \int_{B_R}|u|^q\, dx
    &\le \int_{\mathbb{R}^n} |u|\, |\mathcal{L}_K \varphi_R| \, dx\\
    &= \int_{B_\rho} |u|\, |\mathcal{L}_K \varphi_R| \, dx
      + \int_{B_\rho^c} |u|\, |\mathcal{L}_K \varphi_R| \, dx .
\end{align}
Set
\begin{equation}
J_1 := \int_{B_\rho} |u|\, |\mathcal{L}_K \varphi_R| \, dx,
\qquad
J_2 := \int_{B_\rho^c} |u|\, |\mathcal{L}_K \varphi_R| \, dx .
\end{equation}

Using \eqref{eq:cutoff-lemma}, we estimate the local term by
\begin{align}
    J_1
    \le C R^{-2s} \int_{B_\rho} |u|\, dx ,
\end{align}
and therefore $\lim_{R\to\infty} J_1 = 0$ for fixed $\rho$.

For $J_2$, Hölder's inequality ($\tfrac{n-2s}{n}+\tfrac{2s}{n}=1$) gives
\begin{align}
    J_2
    \le \left(\int_{B_\rho^c} |u|^q \, dx \right)^{\frac{1}{q}}
        \left(\int_{\mathbb{R}^n} |\mathcal{L}_K \varphi_R|^{\frac{n}{2s}}\, dx\right)^{\frac{2s}{n}} .
\end{align}
Applying again \eqref{eq:cutoff-lemma}, we obtain
\begin{align}
    \int_{\mathbb{R}^n} |\mathcal{L}_K \varphi_R|^{\frac{n}{2s}}\, dx
    &\le \int_{B_{2R}} C R^{-n}\, dx
    + C\int_{B_{2R}^c} R^{\frac{n^2}{2s}}|x|^{-\frac{n(n + 2s)}{2s}}\, dx \\
    &\le C + C\int_{2R}^\infty R^{\frac{n^2}{2s}} r^{-\frac{n(n + 2s)}{2s}+n-1}\, dr .
\end{align}
The radial integral is finite and the whole expression is bounded by a constant $C>0$ independent of $R$ and $\rho$.
Consequently,
\begin{equation}
    J_2 \le C \left(\int_{B_\rho^c} |u|^q \, dx \right)^{\frac{1}{q}} .
\end{equation}

Combining the bounds for $J_1$ and $J_2$, we have
\begin{equation}
    \int_{B_R} |u|^q \, dx
    \le J_1 + C \left(\int_{B_\rho^c} |u|^q \, dx \right)^{\frac{1}{q}} .
\end{equation}
Letting $R\to\infty$ and using $J_1\to 0$, we obtain
\begin{equation}
    \int_{\mathbb{R}^n} |u|^q \, dx
    \le C \left(\int_{B_\rho^c} |u|^q \, dx \right)^{\frac{1}{q}} .
\end{equation}
Since $u \in L^q(\mathbb{R}^n)$, the tail integral tends to $0$ as $\rho\to\infty$. Hence
\begin{align}
\int_{\mathbb{R}^n} |u|^q\, dx = 0,
\end{align}
and therefore $u \equiv 0$ in $\mathbb{R}^n$.

\subsection{The low dimensional case: $n \le 2s$}

Finally, we assume $n \le 2s$. In this case, for any $q \ge 1$, the exponent $a$ satisfies
\begin{align}
    a = (n - 2s) - \frac{n}{q} < 0,
\end{align}
since $n - 2s \le 0$ and $n/q > 0$. Therefore, the same iteration argument presented in the subcritical case applies. The sequence of exponents $\gamma_m$ decreases to a negative value, which leads to $S(R) \to 0$ as $R \to \infty$. Thus \(u\equiv0\) in \(\mathbb R^n\), and the proof of Theorem~\ref{thm:main} is complete.

\subsection{Mixed-order and local operators}
\label{sec:mixed-orders}

\begin{proof}[Proof of Corollary~\ref{cor:mixed-orders}]
Let
\begin{align}
\mathcal L
=
c_0(-\Delta)+\sum_{j=1}^m c_j(-\Delta)^{s_j}.
\end{align}
Since each operator in the sum is self-adjoint, one has
\(\mathcal L^\ast=\mathcal L\).

For \(0<s_j<1\), by Lemma~\ref{lem:cutoff-K} applied to the kernel $|z|^{-n-2s_j}$,
\begin{align}
|(-\Delta)^{s_j}\varphi_R(x)|
\le
C_jR^{-2s_j}\mathbf 1_{B_{2R}}(x)
+
C_j\frac{R^n}{|x|^{n+2s_j}}
\mathbf 1_{\mathbb R^n\setminus B_{2R}}(x).
\end{align}
The local term satisfies
\begin{align}
|(-\Delta)\varphi_R|
\le
CR^{-2}\mathbf 1_{B_{2R}}.
\end{align}

For \(R\ge1\), since \(s_j\ge s_1\), we have
\begin{align}
R^{-2s_j}\le R^{-2s_1},
\qquad
R^{-2}\le R^{-2s_1}.
\end{align}
Moreover, if \(|x|\ge2R\), then
\begin{align}
|x|^{-n-2s_j}
\le
|x|^{-n-2s_1}.
\end{align}
Consequently,
\begin{align}
|\mathcal L^\ast\varphi_R(x)|
\le
CR^{-2s_1}\mathbf 1_{B_{2R}}(x)
+
C\frac{R^n}{|x|^{n+2s_1}}
\mathbf 1_{\mathbb R^n\setminus B_{2R}}(x).
\end{align}
Thus \(\mathcal L^\ast\) satisfies the cutoff estimate in
Proposition~\ref{cor:sub} with \(s=s_1\). The conclusion follows from that
proposition.
\end{proof}

%
%
%


\end{document}